\documentclass[11pt,a4paper,twoside]{article}
\usepackage{amssymb}
\usepackage{amsmath}
\usepackage{fancyhdr}
\usepackage{amsthm}
\usepackage{rotate}
\usepackage{graphicx}
\usepackage[T1]{fontenc}

\usepackage{afterpage}

\theoremstyle{plain}
\newtheorem{theorem}{Theorem}
\newtheorem{lemma}{Lemma}

\newtheorem{corollary}{Corollary}

\theoremstyle{definition}
\newtheorem{example}{Example}

\newtheorem{definition}{Definition}

\theoremstyle{remark}

\begin{document}
\afterpage{\rhead[]{\thepage} \chead[\small D.I. Pushkashu       
]{\small  Groupoids which satisfy certain  associative laws} \lhead[\thepage]{} }                  

\begin{center}
\vspace*{2pt}
{\Large \textbf{Groupoids which satisfy certain  associative laws}}\\[3mm]
{\large \textsf{\emph{Dmitry I. Pushkashu }}}\\
\vskip36pt
\textbf{Abstract}
\end{center}
{\footnotesize In this paper we study  properties of left (right) division (cancellative)  groupoids with associative-like  identities, namely, with cyclic associative identity ($x\cdot(y \cdot z) = (z\cdot x) \cdot y$) and  Tarki  ($x \cdot (z\cdot y) = (x\cdot y) \cdot z$ ) identities. }

\footnote{\textsf{2010 Mathematics Subject Classification:} 20N02, 20N05
}

\footnote{\textsf{Keywords:} groupoid, right division groupoid, left division groupoid, right cancellative groupoid, left cancellative groupoid,  cyclic associative identity, quasigroup, abelian group}

\section*{\centerline{Introduction}}

\subsection*{\centerline{Preliminaries}}

\begin{definition} \label{GROU}
{A binary groupoid $(Q, A)$ is understood to be a non-empty set $Q$ together with a binary operation $A$.}
\end{definition}

Let $(Q, \cdot)$ be a groupoid.
The associative law stats that
\begin{equation}x \cdot (y \cdot z) = (x \cdot y) \cdot z \label{(5)}\end{equation}
holds for arbitrary elements $x, y, z \in Q$.

By interchanging the order of the neighboring "factors" in some of the "multiplications" figuring in (\ref{(5)}) it is possible to get 16 equations \cite{HOSSUZU}. In \cite{HOSSUZU} M. Hosszu   reduces these 16 equations to one of the following four equations:
\begin{equation} x\cdot(y \cdot z) = (x\cdot y) \cdot z \label{(6)}\end{equation}
\begin{equation} x\cdot(y \cdot z) = z\cdot(y \cdot x) \; (Grassmann's \; associative \; law) \label{(7)}\end{equation}
\begin{equation} x\cdot(y \cdot z) = y\cdot(x \cdot z) \label{(8)}\end{equation}
\begin{equation} x\cdot(y \cdot z) = (z\cdot x) \cdot y \, (cyclic \; associative \; law) \label{(ASS_CIRCLIC)}\end{equation}

Unfortunately in his article Hosszu gives only two  examples of such reduction of some from these 16  equations to the equations (\ref{(6)})--(\ref{(ASS_CIRCLIC)}) \cite{HOSSUZU}.
\begin{example} The equation $(y\cdot z) \cdot x = (y\cdot x) \cdot z$ is equivalent with $x\ast(z \ast y) = z\ast(x \ast y)$
i.e. with (\ref{(8)}) by introducing the notation $t \ast s = s \cdot t$ \cite{HOSSUZU}.
\end{example}
\begin{example} If we put in the equation
\begin{equation} x \cdot (z\cdot y) = (x\cdot y) \cdot z \, \, (Tarki's \,\, associative\,\,  law) \label{(ASS_TARKI)}\end{equation}
$z = x$ and denote $t = x \cdot y$, we see that the operation satisfies the commutative law: $x \cdot t = t \cdot x$. Hence
 equation (\ref{(ASS_TARKI)}) implies each of the equations (\ref{(6)})-- (\ref{(ASS_CIRCLIC)}), under the only supposition that the set of the elements
$t = x \cdot y$ $(x,y \in Q)$ contains every element of $Q$ \cite{HOSSUZU}.
\end{example}
Notice last condition is true in a left (right) division groupoid.

M. A. Kazim and M. Naseeruddin in \cite{KAZIM} considered the following laws:
\begin{equation} \label{AG}
(x \cdot y) \cdot z = (z \cdot y) \cdot x
\end{equation}
and
\begin{equation} \label{GRAS}
x\cdot(y \cdot z) = z\cdot(y \cdot x)
\end{equation}
As we can see equation (\ref{GRAS}) coincides with  Grassmann's associative  law (\ref{(8)}).
In \cite{KAZIM} identity  (\ref{AG}) is called the left invertible law. Abel-Grassmann groupoids are studied quit intensively \cite{PROT_STEP}.

A groupoid satisfying identity  (\ref{AG}) is called a left almost semigroup and is abbreviated as LA-semigroup \cite{KAZIM}.
Identity $x\cdot(y \cdot z) = z\cdot(y \cdot x)$ (\ref{(7)}) is called right invertible law in \cite{KAZIM}.

A groupoid satisfying the right invertible law (\ref{(7)}) is called a right almost semigroup and is
abbreviated as RA-semigroup \cite{KAZIM}.

\subsection*{\centerline{Definitions}}

Let  $(Q,\cdot)$ be a groupoid.  As usual, the map $L_{a}: Q\rightarrow Q, L_{a}x = a\cdot x$ for all $x\in Q$,
is a left translation of the groupoid  $(Q,\cdot)$ relatively a fixed element $a\in Q$, the map  $R_{a}:
Q\rightarrow Q $, $R_{a} x = x\cdot a$,  is a right translation.
 We give  the following definitions \cite{21,
JKN, VD, 1a, HOP}.

\begin{definition}
{A groupoid $(Q,\cdot)$ is called a left cancelation groupoid, if the following implication fulfilled:
$a\cdot x = a \cdot y \Rightarrow x=y$ for all $a, x, y \in Q$, i.e. translation $L_a$ is an injective map for
any $a\in Q$.}
\end{definition}

\begin{example} \label{EXAMPLE_CANCEL_LEFT} Let   $x\circ y = 1\cdot x + 3\cdot y$ for all  $x, y \in {\mathbb Z} $,
where  $ ({\mathbb Z}, +, \cdot) $ is the ring of integers. It is possible to check that  $ ({\mathbb Z}, \circ)
$ is a left cancelation groupoid.
\end{example}

\begin{definition}
A groupoid $(Q,\cdot)$ is called right  cancelation, if the following implication fulfilled: $x\cdot a = y
\cdot a \Rightarrow x=y$ for all $a, x, y \in G$, i.e. translation $R_a$ is an injective map for any $a\in Q$.
\end{definition}

\begin{definition}
A groupoid $(Q,\cdot)$ is called  a cancelation groupoid, if it is a left and a right cancelation groupoid.
\end{definition}

\begin{definition}
A groupoid $(Q,\cdot)$ is said to be a left (right) division groupoid if the mapping $L_x$ ( $R_x$ ) is
surjective for every $x\in Q$.
\end{definition}

\begin{example} \label{EXAMPLE_CANCEL_LEFT_DIVIZ_RIGHT} Let   $x\circ y = \left[ x/2 \right]  + 3\cdot y$ for all  $x, y \in {\mathbb Z} $,
where  $ ({\mathbb Z}, +, \cdot) $ is the ring of integers. It is possible to check that  $ ({\mathbb Z}, \circ)
$ is  left cancelation right division groupoid.
\end{example}

\begin{definition}
A groupoid $(Q,\cdot)$ is said to be a division groupoid  if it is simultaneously a left and right division
groupoid.
\end{definition}

\begin{definition}
An element $f$ of a groupoid  $(Q,\cdot)$ is called a \textit{left  identity element}, if $f\cdot x = x$ for all $x\in Q$.
An element $e$ of a groupoid  $(Q,\cdot)$ is called a \textit{right  identity element}, if $x\cdot e = x$ for all $x\in Q$.
An element $e$ of a groupoid  $(Q,\cdot)$ is called a \textit{identity element}, if $x\cdot e = x = e\cdot x$ for all $x\in Q$.
\end{definition}

\begin{definition} \label{identity_el_of_groupoid}
A groupoid $(Q,\circ )$ is called a \textit{right  quasigroup} (a \textit{left  quasigroup}) if, for all $a,b
\in Q$, there exists a unique solution $x \in Q$ to the equation $x\circ a = b$ ($a\circ x = b$), i.e. in this
case any right (left) translation of the groupoid $(Q,\circ )$ is a bijective map  of the set $Q$.
\end{definition}

\begin{definition} \label{def2}
A left and right quasigroup is called a \textit{quasigroup}.
\end{definition}

\begin{definition} \label{defLoop2}
A  quasigroup with identity element is called a \textit{loop}.
\end{definition}

In this paper an algebra (or algebraic structure) is a set $A$ together with a collection of operations on $A$.
T. Evans \cite{EVANS_51}  defined  a binary quasigroup as an algebra $(Q, \cdot, /, \backslash)$ with three
binary operations. He has defined the following identities:
\begin{equation} x\cdot(x \backslash y) = y \label{(1)}\end{equation}
\begin{equation}(y / x)\cdot x = y \label{(2)}\end{equation}
\begin{equation}x\backslash (x \cdot y) = y \label{(3)}\end{equation}
\begin{equation}(y \cdot x)/ x = y \label{(4)}\end{equation}

\begin{definition}    An algebra $(Q, \cdot, \backslash, /)$ with identities (\ref{(1)})-- (\ref{(4)}) is called a \textit{quasigroup} \cite{EVANS_51, BIRKHOFF, BURRIS, VD, 1a, HOP}. \label{def3}
\end{definition}

In Mal'tsev terminology  an algebra $(Q, \cdot, \backslash, /)$ with identities (\ref{(1)})-- (\ref{(4)}) is called a \textit{primitive quasigroup} \cite{MALTSEV}.

\begin{definition}    A quasigroup  $(Q, \cdot, \backslash, /)$ with identity $x\backslash x = y\slash y$ is called a \textit{loop} \cite{EVANS_51, BIRKHOFF, BURRIS, VD, 1a, HOP}. \label{def315}
\end{definition}

 A.I. Mal'tsev  named a primitive quasigroup  $(Q, \cdot, \backslash, /)$ with identity $x\backslash x = y\slash y$ a \textit{primitive loop} \cite{MALTSEV}.

 It is proved that definitions of a quasigroup and a primitive quasigroup,  a loop and a primitive loop are equivalent in pairs  \cite{MALTSEV}.

\begin{theorem} \label{ONE_IDEN_LEFT_DIVIS}  \cite{SCERB_07, SCERB_TAB_PUSH_09}
\begin{enumerate}
    \item A groupoid $(Q, \cdot )$ is  a left division groupoid if and only if
 there exists a left cancelation groupoid $(Q, \backslash)$ such that in algebra $(Q, \cdot, \backslash)$
identity  (\ref{(1)}) is fulfilled.
\item A groupoid $(Q, \cdot)$ is a right division groupoid if and only if
there exists a right cancelation groupoid $(Q, \slash)$ such that in algebra   $(Q, \cdot, \slash )$ identity
(\ref{(2)}) is fulfilled.
    \item  A groupoid $(Q, \cdot )$ is  a left cancelation  groupoid if and only if
 there exists a left division groupoid $(Q, \backslash)$ such that in algebra $(Q, \cdot, \backslash)$ identity
(\ref{(3)}) is fulfilled.
    \item  A groupoid $(Q, \cdot )$ is  a right  cancelation  groupoid if and only if
 there exists a right  division groupoid $(Q, \slash)$ such that in algebra $(Q, \cdot, \slash)$ identity
(\ref{(4)}) is fulfilled.
    \end{enumerate}
\end{theorem}

Garret Birkhoff uses the following identities
 \begin{equation}(x / y)\backslash x = y \label{q3}\end{equation}
\begin{equation} y / (x \backslash y) = x \label{q6} \end{equation}

He gives \cite{BIRKHOFF} the following definition of quasigroup $(Q, \cdot, /, \backslash)$.
\begin{definition} \label{(QUASI_MAIN)} An algebra $(Q, \cdot, \backslash, /)$ with identities (\ref{(1)})--(\ref{q6}) is called a \textit{quasigroup}.
\end{definition}

\begin{lemma} \label{left cancellation right division}
In any right division left cancelation  groupoid $(Q, \cdot, /, \backslash)$ is true identity (\ref{q3}) \cite{SPS_10}.
\end{lemma}

\begin{lemma} \label{right cancellation left  division}
In any left  division right cancelation  groupoid $(Q, \cdot, /, \backslash)$ is true identity (\ref{q6}) \cite{SPS_10}.
\end{lemma}

\begin{corollary}
Definitions \ref{(QUASI_MAIN)} and \ref{def3} are equivalent.
\end{corollary}
\begin{proof}
For this aim we should only to prove that identities  (\ref{q3}) and (\ref{q6}) follow from identities
(\ref{(1)})--(\ref{(4)}).
This fact follows from  Lemmas \ref{left cancellation right division} and \ref{right cancellation left  division}.
\end{proof}

By the proving of many results given in this paper we have used Prover 9-Mace 4 \cite{MAC_CUNE_PROV}.

\section*{\centerline{Cyclic associative law}}

In this section we study various division and cancelation groupoids with cyclic associative law.

\subsection*{\centerline{Right division right cancelation groupoid}}

\begin{lemma} \label{CYCL_5}
If a right division groupoid $(Q, \cdot , /)$ satisfies the cyclic associative  law (\ref{(ASS_CIRCLIC)}),
then then it satisfies associative identity (\ref{(5)}).
\end{lemma}
\begin{proof}
From (\ref{(ASS_CIRCLIC)}) renaming  variables as $x \rightarrow y$, $y \rightarrow z$ and $z \rightarrow x$,
we obtain
\begin{equation} \label{150}
(x \cdot y) \cdot z = y \cdot (z \cdot x)
\end{equation}
In (\ref{150}) renaming variables as $x \rightarrow x / z$, $z \rightarrow y$ and $y \rightarrow z$, we obtain
\begin{equation}
((x / z) \cdot z) \cdot y = z \cdot (y \cdot ( x /z))
\end{equation}
Taken in consideration (\ref{(2)}) we get
\begin{equation} \label{151}
x \cdot y = z \cdot (y \cdot (x / z))
\end{equation}
In (\ref{151}) renaming variables as $x \rightarrow z$, $z \rightarrow x$ and swap left and right sides of equation, we obtain
\begin{equation} \label{152}
x \cdot (y \cdot (z / x)) = z \cdot y
\end{equation}
In (\ref{150}) renaming variable $z \rightarrow z / x$, we obtain
\begin{equation} \label{153}
(x \cdot y) \cdot (z / x) = y \cdot ((z / x) \cdot x) \overset{(\ref{(2)})}{=} y \cdot z
\end{equation}
From (\ref{152}) and (\ref{153}) we obtain
\begin{equation} \label{154}
x \cdot (y \cdot z) = z \cdot (x \cdot y)
\end{equation}
From (\ref{154}) and (\ref{(ASS_CIRCLIC)}) we immediately get
\begin{equation}
(z \cdot x) \cdot y = z \cdot (x \cdot y)
\end{equation}

\end{proof}

\begin{lemma} \label{CYCL_6}
If a right division groupoid $(Q, \cdot , /)$ satisfies the cyclic associative  law (\ref{(ASS_CIRCLIC)}), then it is commutative.
\end{lemma}
\begin{proof}
From (\ref{(ASS_CIRCLIC)}) denote  variables as $x \rightarrow y$, $y \rightarrow z$
and $z \rightarrow x$,
we obtain
\begin{equation} \label{160}
(x \cdot y) \cdot z = y \cdot (z \cdot x)
\end{equation}
Rewrite (\ref{(5)}) in the following form
\begin{equation} \label{162}
(x \cdot y) \cdot z = x \cdot (y \cdot z)
\end{equation}
From (\ref{160}) and (\ref{162}) we get
\begin{equation} \label{163}
x \cdot (y \cdot z) = y \cdot (z \cdot x)
\end{equation}
After renaming variable as $y \rightarrow x$, $z \rightarrow y$ and $x \rightarrow z$,
we swap left and right sides of equality, we obtain
\begin{equation} \label{164}
x \cdot (y \cdot z) = z \cdot (x \cdot y)
\end{equation}
We make the following renaming variable in (\ref{162}) $x \rightarrow x / z$, $y \rightarrow z$
and $z \rightarrow y$, we get
\begin{equation} \label{165}
((x /z) \cdot z) \cdot y \overset{(\ref{(2)})}{=} x \cdot y = (x / z) \cdot (z \cdot y)
\end{equation}
In (\ref{165}) we renamed variables as $z \rightarrow y$ and $y \rightarrow z$,
we get
\begin{equation} \label{166}
(x / y) \cdot (y \cdot z) =  x \cdot z
\end{equation}
In (\ref{164}) renaming variable as $x \rightarrow x /y$
\begin{equation} \label{167}
(x / y) \cdot (y \cdot z) = z \cdot ((x / y) \cdot y) \overset{(\ref{(2)})}{=} z \cdot x
\end{equation}
From (\ref{166}) and (\ref{167}) we immediately obtain
\begin{equation}
x \cdot z = z \cdot x
\end{equation}
\end{proof}

\begin{lemma} \label{CYCL_11}
If a right division right cancelation groupoid $(Q, \cdot , /)$ satisfies the cyclic associative law (\ref{(ASS_CIRCLIC)}), then $z = (x / x) \cdot z$ for all $x, z\in Q$.
\end{lemma}
\begin{proof}
In view of results from Lemma \ref{CYCL_5} and \ref{CYCL_6} we obtained that right division groupoid
satisfied  associative (\ref{(5)}) and it is commutative ($x \cdot y = y \cdot x$).
Rewrite (\ref{(5)}) as follows
\begin{equation} \label{CYCL11_1}
(x \cdot y) \cdot z = x \cdot (y \cdot z)
\end{equation}
From (\ref{(2)}) and commutative law we obtain
\begin{equation} \label{CYCL11_2}
x \cdot (y / x) = y
\end{equation}
Considering (\ref{CYCL11_1}), denote variables $x \rightarrow y$ and $y \rightarrow x$, and in view of commutative law, we obtain
\begin{equation} \label{CYCL11_3}
(x \cdot y) \cdot z = y \cdot (x \cdot z)
\end{equation}
From (\ref{CYCL11_1}) and (\ref{CYCL11_3}) we have
\begin{equation} \label{CYCL11_4}
x \cdot (y \cdot z) = y \cdot (x \cdot z)
\end{equation}
Considering (\ref{CYCL11_1}), denote variable $y \rightarrow (y / x)$ we obtain
\begin{equation} \label{CYCL11_5}
(x \cdot (y / x)) \cdot z \overset{(\ref{CYCL11_2})}{=} y \cdot z = x \cdot ((y / x) \cdot z)
\end{equation}
Considering (\ref{CYCL11_5}), denote variable $y \rightarrow x$ we obtain
\begin{equation} \label{CYCL11_6}
x \cdot z = x \cdot ((x / x) \cdot z)
\end{equation}
From (\ref{CYCL11_6}) and (\ref{(4)}) we immediately obtain
\begin{equation} \label{CYC_999}
z = (x / x) \cdot z
\end{equation}
\end{proof}

\begin{lemma} \label{CYCL_12}
If a right division right cancelation groupoid $(Q, \cdot , /)$ satisfies the cyclic associative law (\ref{(ASS_CIRCLIC)}), then $x / x =  y / y$ for all $x, y\in Q$.
\end{lemma}
\begin{proof}
If we put in identity \ref{(4)} $x \rightarrow x \slash x$ then we have $((x / x) \cdot y) / y = x / x$, and using equality (\ref{CYC_999}), we obtain $x /x = y / y$.
\end{proof}
In fact from previous lemmas it  follows that right division right cancelation groupoid with cyclic associative law has an identity element.

\begin{theorem} \label{CYCL_13}
If a right division right cancelation groupoid $(Q, \cdot, /)$ satisfies the cyclic associative law (\ref{(ASS_CIRCLIC)}),
then it is a commutative group relative to the operation $\cdot$ and it satisfies any from identities (\ref{(6)}) -- ( \ref{(ASS_CIRCLIC)}).
\end{theorem}
\begin{proof}
By Lemma \ref{CYCL_6} it is a commutative groupoid.
Then by Lemmas \ref{CYCL_11} and \ref{CYCL_12} right division groupoid $(Q, \cdot, /)$ which satisfies  the cyclic associative  law has an identity element.

Therefore $(Q, \cdot, /)$ is a  division  cancelation groupoid with an identity element, i.e. it is a commutative loop, it is a commutative group.
\end{proof}

\subsection*{\centerline{Right division left cancelation groupoid}}

\begin{lemma} \label{LOOP_1}
If a right division left cancelation groupoid $(Q, \cdot, /, \backslash)$ satisfies the cyclic associative  law (\ref{(ASS_CIRCLIC)}), then it satisfies identity $y / y =x \backslash x$.
\end{lemma}
\begin{proof}
In identity $ x \cdot (y \cdot z) = (z \cdot x) \cdot y$ after renaming of variables we obtain
\begin{equation} \label{L62}
(y \cdot x) \cdot z = x \cdot (z \cdot y)
\end{equation}
 From identity (\ref{L62}) using identity (\ref{(3)})   we have
\begin{equation} \label{L82}
x \backslash ((y \cdot x) \cdot z) = z \cdot y
\end{equation}
If $y \cdot x = t$, then $t/ x = y$,  since $(t / x) \cdot x = t$ (identity (\ref{(2)})). Then we can re-write identity (\ref{L82})
in the following form
\begin{equation}
x \backslash (y \cdot z) = z \cdot (y / x) \label{L92}
\end{equation}
If we put in (\ref{L92}) $x = y$ and take into consideration that  $x \backslash (x \cdot z) = z$
(identity (\ref{(3)})), then we obtain $z = z \cdot (x / x)$. Changing the letter $z$ by letter $x$,  and the letter $x$ by letter $y$ further we have
\begin{equation}
x * (y / y) = x  \label{L102}
\end{equation}
Using  identity  (\ref{(2)}) from identity  (\ref{L102})
we obtain
\begin{equation}
x / (y / y) = x \label{L112}
\end{equation}
From  identity  (\ref{L112}) we have
\begin{equation}
(x / (y / y))\backslash x  = x\backslash x \label{L122}
\end{equation}

By Corollary \ref{left cancellation right division}  any right division left cancelation groupoid satisfies identity
(\ref{q3}).
If we apply to the left side of identity  (\ref{L122})  identity (\ref{q3}), then finally
we have
\begin{equation}
y / y =x \backslash x \label{quasigroup_IDENT}
\end{equation}
\end{proof}

\begin{lemma} \label{LOOP_11}
If a right division left cancelation groupoid $(Q, *, /, \backslash)$ satisfies the cyclic associative  law (\ref{(ASS_CIRCLIC)}), then it has an identity element.
\end{lemma}
\begin{proof}
By Lemma \ref{LOOP_1} groupoid $(Q, *, /, \backslash)$ satisfies  identity $y / y =x \backslash x$. From this identity  we have $(y / y) * y  = (x \backslash x) * y$. Applying identity   $(x / y) * y = x$ (identity (\ref{(2)}))  to the left side of the last identity  we have $(x \backslash  x) * y = y$.

Then groupoid
$(Q, *, /, \backslash)$ has at least one left identity element $f$. Further we have $ x \backslash x = y / y = y \backslash y$, i.e. $x \backslash x = y \backslash y$. The last proves that we have unique left identity element $f$. From commutativity of the groupoid $(Q, *, /, \backslash)$ (Lemma \ref{CYCL_6}) we obtain, that this groupoid has two-sided identity element.
\end{proof}

\begin{theorem} \label{L_1}
If a right division left cancelation groupoid $(Q, \cdot, /, \backslash)$ satisfies  the cyclic associative  law (\ref{(ASS_CIRCLIC)}), then it is a commutative group relative to the operation $\cdot$ and it satisfies any from identities (\ref{(6)}) -- ( \ref{(ASS_CIRCLIC)}).
\end{theorem}
\begin{proof}
By Lemma \ref{LOOP_11} right division left cancelation groupoid $(Q, \cdot, /, \backslash)$ which satisfies  the cyclic associative  law has an identity element.
Thus $(Q, \cdot, /, \backslash)$ is a  division  cancelation groupoid with an identity element, i.e. it is a commutative loop, a commutative group (Lemmas \ref{CYCL_5} and \ref{CYCL_6}).

It easy to see that any commutative group satisfies identities (\ref{(7)}) and  (\ref{(8)}).
\end{proof}

\subsection*{\centerline{Left division left cancelation groupoid}}

\begin{lemma} \label{CYCL_1}
If a left division groupoid $(Q, \cdot , \backslash)$ satisfies the cyclic associative  law (\ref{(ASS_CIRCLIC)}),
then it satisfies and ordinary associative law $(x \cdot y) \cdot z = x \cdot (y \cdot z)$ (\ref{(5)}).
\end{lemma}
\begin{proof}
If in identity (\ref{(ASS_CIRCLIC)}) we rename variables $x \rightarrow y$, $y \rightarrow z$ and $z \rightarrow x$,
then we obtain
\begin{equation} \label{10}
(x \cdot y) \cdot z = y \cdot (z \cdot x)
\end{equation}
If in equality (\ref{10}) we denote variables $y \rightarrow z \backslash x$, $z \rightarrow y$ and $x \rightarrow z$, then we obtain
\begin{equation}
(z \cdot (z \backslash x)) \cdot y = (z \backslash x) \cdot (y \cdot x)
\end{equation}
In view of equality (\ref{(1)}) we obtain the following identity
\begin{equation} \label{11}
x \cdot y = (z \backslash x) \cdot (y \cdot z)
\end{equation}
In (\ref{11}) denote variables $x \rightarrow y$, $y \rightarrow z$ and $z \rightarrow x$ and swap left and right sides of equality. Then  we have
\begin{equation} \label{12}
(x \backslash y) \cdot (z \cdot x) = y \cdot z
\end{equation}
In equality  (\ref{12}) we denote variable $x \rightarrow u$ and  we get
\begin{equation}
(u \backslash y) \cdot (z \cdot u) = y \cdot z
\end{equation}
Further we multiply both sides of equality by $x$ and  obtain the following
\begin{equation} \label{13}
(y \cdot z) \cdot x = ((u \backslash y) \cdot (z \cdot u)) \cdot x
\end{equation}
In (\ref{13}) denote variables $u \backslash y \rightarrow x$, $z \cdot u \rightarrow y$ and
$x \rightarrow z$ and considering (\ref{10}), we swap left and right sides of equation and obtain
\begin{equation} \label{14}
(x \cdot y) \cdot (z \cdot (y \backslash u)) = (u \cdot x) \cdot z
\end{equation}
In (\ref{12}) denote variable $z \rightarrow (u \cdot z)$ and in view of (\ref{10}), we have
\begin{equation} \label{15}
(x \backslash y) \cdot (z \cdot (x \cdot u)) = y \cdot (u \cdot z)
\end{equation}
If we assume that in the equality (\ref{14}) variable $x$ is equal to $y$ ($x = y$), then we get
\begin{equation} \label{16}
(x \cdot y) \cdot z = (z \cdot x) \cdot y
\end{equation}
From equalities (\ref{16}) and (\ref{(ASS_CIRCLIC)}) we immediately obtain that
\begin{equation}
(x \cdot y) \cdot z = x \cdot (y \cdot z)
\end{equation}
\end{proof}

\begin{lemma} \label{CYCL_2}
If a left division groupoid $(Q, \cdot , \backslash)$ satisfies the cyclic associative  law (\ref{(ASS_CIRCLIC)}), then it is a commutative groupoid.
\end{lemma}
\begin{proof}
In (\ref{(ASS_CIRCLIC)}) denote  variables $x \rightarrow y$, $y \rightarrow z$ and $z \rightarrow x$,
we obtain
\begin{equation} \label{20}
(x \cdot y) \cdot z = y \cdot (z \cdot x)
\end{equation}
By  Lemma \ref{CYCL_1} $(Q, \cdot , \backslash)$ is a left division groupoid with identity (\ref{(ASS_CIRCLIC)}) satisfies the associative law (\ref{(5)}).

We  rewrite (\ref{(5)}) as follows
\begin{equation} \label{21}
(x \cdot y) \cdot z = x \cdot (y \cdot z)
\end{equation}
Comparing with (\ref{20}) we can write
\begin{equation} \label{22}
x \cdot (y \cdot z) = y \cdot (z \cdot x)
\end{equation}
From (\ref{(5)}) taking into consideration (\ref{21}) we obtain
\begin{equation} \label{23}
x \cdot (y \cdot z) = z \cdot (x \cdot y)
\end{equation}
If in identity (\ref{21}) we denote variable $y \rightarrow (x \backslash y)$, we have
\begin{equation} \label{241}
(x \cdot (x \backslash y)) \cdot z = x \cdot ((x \backslash y) \cdot z)
\end{equation}
and in view of (\ref{(5)}) we obtain
\begin{equation}
y \cdot z = x \cdot ((x \backslash y) \cdot z)
\end{equation}
If in (\ref{241}) we denote variables $y \rightarrow x$, $z \rightarrow y$ and $x \rightarrow z$, we get
\begin{equation} \label{25}
x \cdot y = z \cdot ((z \backslash x) \cdot y)
\end{equation}
In (\ref{23}) we denote variable $z \rightarrow (y \backslash z)$ , we have
\begin{equation} \label{26}
x \cdot (y \cdot (y \backslash z)) \overset{(\ref{(1)})}{=} x \cdot z = (y \backslash z) \cdot (x \cdot y)
\end{equation}
In view of (\ref{23}) and (\ref{26}) we obtain
\begin{equation} \label{27}
x \cdot y = z \cdot ((z \backslash y) \cdot x)
\end{equation}
Considering (\ref{25}), we swap left and right sides of equality and in view of substitutions $z \rightarrow x$ and $x \rightarrow z$, we have
\begin{equation} \label{28}
z \cdot ((z \backslash y) \cdot x) = y \cdot x
\end{equation}
From (\ref{27}) and (\ref{28}) we immediately obtain
\begin{equation}
x \cdot y = y \cdot x
\end{equation}
\end{proof}

\begin{lemma} \label{CYCL_3}
If a left division left cancelation groupoid $(Q, \cdot , \backslash)$ satisfies the cyclic associative law (\ref{(ASS_CIRCLIC)}), then $x  = (y\backslash y) x$ for all $x, y \in Q$.
\end{lemma}
\begin{proof}
If in identity of associativity $(xy)z = x(yz)$ we change $y$ by $x\backslash y$ then we have
\begin{equation} \label{P1}
x(x\backslash y) \cdot z  = x \cdot (x\backslash y)z
\end{equation}
 and after application to the left side of  equality (\ref{P1})  identity (\ref{(1)}) we obtain
\begin{equation} \label{P1}
y z  = x \cdot (x\backslash y)z
\end{equation}
From (\ref{P1}) we have
\begin{equation} \label{P2}
x\backslash (y z)  = x\backslash (x \cdot (x\backslash y)z)
\end{equation}
 after application to the right  side of  equality (\ref{P2})  identity (\ref{(3)}) we obtain
\begin{equation} \label{P3}
x\backslash (y z)  = (x\backslash y)z
\end{equation}
If we put in equality (\ref{P2}) $x=y$, then using identity (\ref{(3)}) we obtain
\begin{equation} \label{P3}
z  = (x\backslash x)z
\end{equation}
\end{proof}

\begin{lemma} \label{CYCL_2233}
If a left division left cancelation groupoid $(Q, \cdot, \backslash)$ satisfies the cyclic associative law (\ref{(ASS_CIRCLIC)}), then $x\backslash x = y\backslash y$ for all $x, y \in Q$.
\end{lemma}
\begin{proof}
From identity  $x \backslash  (x \cdot y) = y$ using commutativity (Lemma \ref{CYCL_2}) we have $x \backslash (y \cdot x) = y$. If we change in the last identity $y$ by $y\backslash y$, then we have $x \backslash ((y\backslash y) \cdot x) = y\backslash y$, $x \backslash x = y\backslash y$.
\end{proof}

\begin{theorem} \label{CYCL_7}
If a left division left cancelation groupoid $(Q, \cdot, \backslash)$ satisfies the cyclic associative law (\ref{(ASS_CIRCLIC)}),
then it is a commutative group relative to the operation $\cdot$ and it satisfies any from identities (\ref{(6)}) -- ( \ref{(ASS_CIRCLIC)}).
\end{theorem}
\begin{proof}
By Lemma \ref{CYCL_2} left division left cancelation groupoid (left quasigroup) $(Q, \cdot, \backslash)$ which satisfies  the cyclic associative  law is  commutative  relative to the operation $\cdot$.
By Lemma \ref{CYCL_2233} the groupoid has an identity element.

Therefore $(Q, \cdot, \backslash)$ is a  division  cancelation groupoid with an identity element, i.e. it is a commutative loop, i.e. abelian group.
\end{proof}
\begin{theorem} \label{CYCL_77}
If a left division right cancelation groupoid $(Q, \cdot, \backslash, /)$ satisfies the cyclic associative law (\ref{(ASS_CIRCLIC)}),
then it is a commutative group relative to the operation $\cdot$ and it satisfies any from identities (\ref{(6)}) -- ( \ref{(ASS_CIRCLIC)}).
\end{theorem}
\begin{proof}
By Lemmas \ref{CYCL_1} and \ref{CYCL_2} groupoid $(Q, \cdot, \backslash, /)$ is a commutative semigroup.
Prove that this groupoid has a unique identity element.

If we change $y \rightarrow x\backslash y$ in the identity of associativity $(xy) z = x (yz)$, then we have
$(x(x\backslash y))z = x((x\backslash y)z)$. If we apply identity (\ref{(1)}) to the left side of the last identity, then we obtain
\begin{equation} \label{CYCL_134}
x((x\backslash y)z) = yz
\end{equation}
From commutativity ($xy = yx$) and right cancellative identity   ($(xy) / y = x$) we obtain
\begin{equation} \label{ten_CYCL}
(yx)/y = x
\end{equation}
From identity (\ref{(1)}) we have $(x(x\backslash y)) / x = y/ x$. If we apply to the left side of the last identity
(\ref{ten_CYCL}), then we obtain
\begin{equation} \label{72}
x/y = y \backslash x
\end{equation}

Using (\ref{72}) we can rewrite (\ref{CYCL_134}) in the form
\begin{equation} \label{15_CYCL}
x((y/ x)z) = yz
\end{equation}
From  (\ref{15_CYCL}) we have
\begin{equation} \label{14_CYCL}
(x((y/ x)z)) / x = (yz)/x
\end{equation}
If we apply (\ref{ten_CYCL}) to the left side of equality (\ref{14_CYCL}), we have
\begin{equation} \label{154_CYCL}
(y / x)z  = (yz)/x
\end{equation}
If we take in (\ref{154_CYCL}) $x=y$, then
\begin{equation} \label{16_CYCL}
(x / x)z  = (xz)/x \overset {(\ref{ten_CYCL})}{=} z
\end{equation}
Since $(Q, \cdot, \backslash, /)$ is commutative, further we have
\begin{equation} \label{16_CYCL}
z (x / x)= z
\end{equation}
If we change in (\ref{16_CYCL}) $z\rightarrow z/z$, then we have $(z/z) (x / x)= z/z$. If we apply to the left side of the last identity equality (\ref{16_CYCL}), then
\begin{equation} \label{156_CYCL}
x / x= z/z
\end{equation}
From (\ref{16_CYCL}), (\ref{156_CYCL}) and commutativity of
groupoid $(Q, \cdot, \backslash, /)$ we conclude that $(Q, \cdot, \backslash, /)$ has a unique two-sided identity element and it is a commutative group.
\end{proof}

\section*{\centerline{Tarki associative law}}
\subsection*{\centerline{Left division left cancelation groupoid}}
\begin{lemma} \label{TARKI_1}
If a left division groupoid $(Q, \cdot , \backslash)$ satisfies the Tarki associative  law (\ref{(ASS_TARKI)}), then it is a commutative groupoid.
\end{lemma}
\begin{proof}
In (\ref{(ASS_TARKI)}) denote variables $z \rightarrow y$ and $y \rightarrow z$. We obtain
\begin{equation} \label{T_1}
(x \cdot y) \cdot z = x \cdot (z \cdot y)
\end{equation}
In (\ref{(ASS_TARKI)}) denote variables $x \rightarrow z$ and $z \rightarrow z \backslash x$, we have
\begin{equation} \label{T_2}
z \cdot (y \cdot (z \backslash x)) = (z \cdot (z \backslash x)) \cdot y \overset{(\ref{(1)})}{=} x \cdot y
\end{equation}
From (\ref{T_2}) we get
\begin{equation}  \label{T_3}
x \cdot (y \cdot (x \backslash z)) = z \cdot y
\end{equation}
In  (\ref{T_3}) denote variables $x \rightarrow y$, $z \rightarrow x$, we obtain
\begin{equation}
y \cdot (y \cdot (y \backslash x)) \overset{(\ref{(1)})}{=} y \cdot x = x \cdot y
\end{equation}
\end{proof}

\begin{lemma} \label{TARKI_2}
If a left division groupoid $(Q, \cdot , \backslash)$ satisfies the Tarki associative  law (\ref{(ASS_TARKI)}),
then it satisfies associative law (\ref{(5)}).
\end{lemma}
\begin{proof}
In  (\ref{(ASS_TARKI)}) denote variables $z \rightarrow y$ and $y \rightarrow z$. We obtain
\begin{equation} \label{T2_1}
(x \cdot y) \cdot z = x \cdot (z \cdot y)
\end{equation}
As follows from Lemma \ref{TARKI_1} left division groupoid satisfies commutative law.
We used commutative law for the right side of  (\ref{T2_1}) and  immediately obtain
\begin{equation}
(x \cdot y) \cdot z = x \cdot (y \cdot z)
\end{equation}
\end{proof}

\begin{lemma} \label{TARKI_6}
If a left division left cancelation groupoid $(Q, \cdot , \backslash)$ satisfies the Tarki associative  law (\ref{(ASS_TARKI)}),
then it satisfies $x = x \cdot (y \backslash y)$ and $x = (y \backslash y) \cdot x$.
\end{lemma}
\begin{proof}
As follows from Lemma \ref{TARKI_1}, left division groupoid with Tarki associativity law (\ref{(ASS_TARKI)}) is a commutative groupoid, it is means that
\begin{equation} \label{COMM}
x \cdot y = y \cdot x
\end{equation}
 for all $x, y \in Q$.

>From Lemma \ref{TARKI_2}, it follows that the groupoid $(Q, \cdot, \backslash)$ is associative (\ref{(5)}).
We rewrite identity  (\ref{(5)}) as follows
\begin{equation} \label{T6_1}
(x \cdot y) \cdot z = x \cdot (y \cdot z)
\end{equation}
Using (\ref{COMM}) and (\ref{(5)}) we write  (\ref{T6_1}) as follows
\begin{equation}
(x \cdot y) \cdot z = x \cdot (y \cdot z) \overset{(\ref{COMM})}{=} (y \cdot z) \cdot x  \overset{(\ref{(5)})}{=} y \cdot (z \cdot x) \overset{(\ref{COMM})}{=} y \cdot (x \cdot z)
\end{equation}
As a result we have
\begin{equation} \label{T6_2}
(x \cdot y) \cdot z = y \cdot (x \cdot z)
\end{equation}
Further  we rewrite left side of (\ref{T6_2}) using (\ref{(5)}), as follows
\begin{equation} \label{T6_3}
(x \cdot y) \cdot z \overset{(\ref{(5)})}{=}  x \cdot (y \cdot z) = y \cdot (x \cdot z)
\end{equation}
In (\ref{T6_1}) denote variable $y \rightarrow (x \backslash y)$  and we obtain
\begin{equation} \label{T6_4}
(x \cdot (x \backslash y)) \cdot z \overset{(\ref{(1)})}{=} y \cdot z = x \cdot ((x \backslash y) \cdot z)
\end{equation}
In (\ref{T6_4}) rename variables $x \rightarrow z$, $y \rightarrow x$ and $z \rightarrow y$. We have
\begin{equation} \label{T6_5}
x \cdot y = z \cdot ((z \backslash x) \cdot y)
\end{equation}
In (\ref{T6_3}) rename variables $y \rightarrow z \backslash y$ we obtain
\begin{equation} \label{T6_6}
x \cdot ((z \backslash y) \cdot z) = (z \backslash y) \cdot (x \cdot z)
\end{equation}
For the left side of the (\ref{T6_6}) we get
\begin{equation} \label{T6_7}
x \cdot ((z \backslash y) \cdot z) \overset{(\ref{COMM})}{=} x \cdot y
\end{equation}
For the right side of the (\ref{T6_6}) we get
\begin{equation} \label{T6_8}
(z \backslash y) \cdot (x \cdot z) \overset{(\ref{COMM})}{=} (z \backslash y) \cdot (z \cdot x) \overset{(\ref{T6_3})}{=} z \cdot (x \cdot (z \backslash y))
\end{equation}
From (\ref{T6_7}) and (\ref{T6_8}) we immediately obtain
\begin{equation} \label{T6_9}
x \cdot y = z \cdot (x \cdot (z \backslash y))
\end{equation}
Considering (\ref{T6_9}) denote variables $z \rightarrow x$, $x \rightarrow y$
and $y \rightarrow z $, we get
\begin{equation} \label{T6_10}
x \cdot (y \cdot (x \backslash z)) = y \cdot z
\end{equation}
From (\ref{T6_4}) and (\ref{T6_10}) we have
\begin{equation} \label{T6_11}
x \cdot ((x \backslash y) \cdot z) = x \cdot (y \cdot (x \backslash z))
\end{equation}
By appling left cancelation (\ref{(3)}) to the (\ref{T6_11}), we get
\begin{equation} \label{T6_12}
(x \backslash y) \cdot z = y \cdot (x \backslash z)
\end{equation}
Let's $x = y$ in (\ref{T6_12}), we get
\begin{equation}
(y \backslash y) \cdot z = y \cdot (y \backslash z) \overset{(\ref{(3)})}{=} z
\end{equation}
and finally we obtain
\begin{equation}
z = (y \backslash y) \cdot z \overset{(\ref{COMM})}{=} z \cdot (y \backslash y)
\end{equation}
\end{proof}

\begin{lemma} \label{TARKI_7}
If a left division left cancelation groupoid $(Q, \cdot, \backslash)$ satisfies the Tarki associative  law (\ref{(ASS_TARKI)}), then $x \backslash x = y \backslash y$ for all $x, y \in Q$.
\end{lemma}
\begin{proof}
From identity  $x \backslash  (x \cdot y) = y$ using commutativity (Lemma \ref{TARKI_1}) we have
$x \backslash (y \cdot x) = y$. If we change in the last identity $y$ by $y\backslash y$, then we have
$x \backslash ((y\backslash y) \cdot x) = y\backslash y$, $x \backslash x = y\backslash y$.
\end{proof}

\begin{theorem} \label{TARKI_THEOR_1}
If a left division left cancelation groupoid $(Q, \cdot, \backslash)$ satisfies the Tarki associative  law (\ref{(ASS_TARKI)}), then it is a commutative group relative to the operation $\cdot$ and it satisfies any from identities (\ref{(6)}) -- ( \ref{(ASS_CIRCLIC)}).
\end{theorem}
\begin{proof}
By Lemmas \ref{TARKI_1} and \ref{TARKI_2} left division left cancelation groupoid (left quasigroup) $(Q, \cdot, \backslash)$ which satisfies the Tarki associative  law is a commutative semigroup relative to the operation $\cdot$.
By Lemmas \ref{TARKI_6} and  \ref{TARKI_7} the groupoid has a unique identity element.

Therefore $(Q, \cdot)$ is a commutative group and it satisfies any from identities (\ref{(6)}) -- ( \ref{(ASS_CIRCLIC)}).
\end{proof}

\subsection*{\centerline{Left division right cancelation groupoid}}

\begin{lemma} \label{TARKI_8}
If a left division right cancelation groupoid $(Q, \cdot , \backslash, /)$ satisfies the Tarki associative  law (\ref{(ASS_TARKI)}), then it satisfies $x \backslash y = y / x$.
\end{lemma}
\begin{proof}
Let's start from equality (\ref{(4)}) where we rename variable $x \rightarrow y$ and $y \rightarrow x$. We have
\begin{equation} \label{T87_1}
 x = (x \cdot y) / y \overset{commutativity}{=} (y \cdot x) / y
\end{equation}
In equality (\ref{T87_1}) we rename variable $x \rightarrow y$ and $y \rightarrow x$, we get
\begin{equation} \label{T87_2}
(x \cdot y) / x = y
\end{equation}
Considering (\ref{(1)}) and (\ref{T87_2}) we can write as follows
\begin{equation} \label{T87_3}
x \cdot (x \backslash y) = (x \cdot y) / x
\end{equation}
and
\begin{equation} \label{T87_4}
(x \cdot (x \backslash y)) / x  = ((x \cdot y) / x) / x
\end{equation}
\begin{equation} \label{T87_5}
(x \cdot (x \backslash y)) / x  = x \backslash y
\end{equation}
\begin{equation} \label{T87_6}
((x \cdot y) / x) / x = y / x
\end{equation}
From (\ref{T87_5}) and (\ref{T87_6}) we obtain
\begin{equation} \label{T87_7}
x \backslash y = y / x
\end{equation}
Let's $ x = y$, then we obtain
\begin{equation}
x / x = x \backslash x
\end{equation}
\end{proof}

\begin{lemma} \label{TARKI_9}
If a left division right cancelation groupoid $(Q, \cdot , \backslash, /)$ satisfies the Tarki associative  law (\ref{(ASS_TARKI)}), then it satisfies $y = (x / x) \cdot y$.
\end{lemma}
\begin{proof}
Let's write (\ref{(5)}) as follows
\begin{equation} \label{T9_1}
(x \cdot y) \cdot z  = x \cdot (y \cdot z)
\end{equation}
From Lemma \ref{TARKI_8} we have
\begin{equation} \label{T9_3}
x \backslash y = y / x
\end{equation}
From (\ref{T9_3}) and (\ref{(1)}) we get
\begin{equation} \label{T9_4}
x \cdot (y / x) = y
\end{equation}
Taking into consideration  (\ref{(4)}) and commutative law we can write
\begin{equation} \label{T9_5}
(y \cdot x) / y = x
\end{equation}
In equality (\ref{T9_5}) rename variables $x \rightarrow y$ and $y \rightarrow x$, we have
\begin{equation} \label{T9_6}
(x \cdot y) / x = y
\end{equation}
From (\ref{T9_4}) and (\ref{T9_1}) we obtain
\begin{equation} \label{T9_7}
x \cdot y = z \cdot ((x / z) \cdot y)
\end{equation}
Rename variables $x \rightarrow y$, $y \rightarrow z$ and $z \rightarrow x$ in equality (\ref{T9_7})
 as follows
\begin{equation} \label{T9_8}
x \cdot ((y / x) \cdot z) = y \cdot z
\end{equation}
We "multiply" both sides of equality  (\ref{T9_8}) and obtain
\begin{equation} \label{T9_9}
(x \cdot ((y / x) \cdot z)) / x = (y \cdot z) / x
\end{equation}

Using (\ref{T87_2}) in (\ref{T9_9}),   renaming  variables $y \rightarrow x$, $z \rightarrow y$ and $x \rightarrow z$, and swapping  left and right sides of equality (\ref{T9_9}), we have
\begin{equation} \label{T9_9}
(x \cdot y) / z = (x / z) \cdot y
\end{equation}
If we put  in equality (\ref{T9_9}) $z = x$, then  we obtain
\begin{equation} \label{T9_10}
(x \cdot y) / x \overset{(\ref{T9_6})}{=} y = (x / x) \cdot y
\end{equation}
\end{proof}

\begin{lemma} \label{TARKI_10}
If a left division left cancelation groupoid $(Q, \cdot, \backslash, /)$ satisfies the Tarki associative  law (\ref{(ASS_TARKI)}), then $x / x = y / y$ for all $x, y \in Q$.
\end{lemma}
\begin{proof}
From identity  $x \cdot (y / y) = x $ using commutativity (Lemma \ref{TARKI_1}) we have
$(y / y) \cdot x = x$.
From the last equality we have
$((y / y) \cdot x) / x = x / x$, $y / y = x / x$.
\end{proof}

\begin{theorem} \label{TARKI_THEOR_2}
If a left division right cancelation groupoid $(Q, \cdot, \backslash, /)$ satisfies the Tarki associative  law (\ref{(ASS_TARKI)}), then it is a commutative group relative to the operation $\cdot$ and it satisfies any from identities (\ref{(6)}) -- ( \ref{(ASS_CIRCLIC)}).
\end{theorem}
\begin{proof}
By Lemma \ref{TARKI_1} left division right cancelation groupoid $(Q, \cdot, \backslash, /)$ which satisfies the Tarki associative  law is  commutative  relative to the operation $\cdot$.
By Lemma \ref{TARKI_10} the groupoid has an identity element.

Then $(Q, \cdot, \backslash, /)$ is a  commutative group relative to the operation $\cdot$.
\end{proof}

\subsection*{\centerline{Right division right cancelation groupoid}}

\begin{lemma} \label{TARKI_5}
If a right division right cancelation groupoid $(Q, \cdot , /)$ satisfies the Tarki associative  law (\ref{(ASS_TARKI)}),
then it satisfies the identity  $x \cdot (y / y) = x$.
\end{lemma}
\begin{proof}
We  rewrite identity  (\ref{(ASS_TARKI)})  using renaming variables as $y \rightarrow z$,
$z \rightarrow y$, as follows
\begin{equation} \label{T5_1}
(x \cdot y) \cdot z = x \cdot (z \cdot y)
\end{equation}
If in   (\ref{T5_1}) we rename variables as $x \rightarrow (x / z)$, $y \rightarrow z$ and $z \rightarrow y$,
we obtain
\begin{equation} \label{T5_2}
((x / z) \cdot z) \cdot y \overset{(\ref{(2)})}{=} x y = (x / z) \cdot (y \cdot z)
\end{equation}
If we  rewrite (\ref{T5_2}) using renaming variables as $z \rightarrow y$,
$y \rightarrow z$, we have
\begin{equation} \label{T5_3}
(x / y) \cdot (z \cdot y) = x \cdot z
\end{equation}
In (\ref{T5_3}) we rename variable as $z \rightarrow (z / y)$. We get
\begin{equation} \label{T5_4}
(x / y) \cdot ((z / y) \cdot y) \overset{(\ref{(2)})}{=} (x / y) \cdot z = x \cdot (z / y)
\end{equation}
Let's in (\ref{T5_4}) $z = y$. Then we obtain
\begin{equation}  \label{T5_5}
(x / y) \cdot y \overset{(\ref{(2)})}{=} x = x \cdot (y / y)
\end{equation}
\end{proof}

\begin{lemma} \label{TARKI_51}
If a right division right cancelation groupoid $(Q, \cdot , /)$ satisfies  Tarki  law (\ref{(ASS_TARKI)}), then it is associative (\ref{(5)}).
\end{lemma}
\begin{proof}
Using Tarki  law we can rewrite associativity in the following form
\begin{equation} \label{90}
x \cdot (z \cdot y) = x \cdot (y \cdot z)
\end{equation}
 Then, we prove this lemma, if we prove identity (\ref{90}).

>From (\ref{(ASS_TARKI)}) and (\ref{(4)}) we have
\begin{equation} \label{91}
(x \cdot (y \cdot z)) / y = x \cdot z
\end{equation}
Indeed, $(x \cdot y) \cdot z = x \cdot (z \cdot y)$, $((x \cdot y) \cdot z)/z = (x \cdot (z \cdot y))/z$, $x \cdot y = (x \cdot (z \cdot y))/z$.

>From (\ref{91}) and (\ref{(ASS_TARKI)}) we have
\begin{equation}
(x \cdot (y \cdot (z \cdot u)) / y = (x \cdot u) \cdot z
\end{equation}
Indeed, we can change  in (\ref{91}) $x \rightarrow xu$ and apply to the left side of obtained identity Tarki associative law (\ref{(ASS_TARKI)}).

>From (\ref{91}) and (\ref{(ASS_TARKI)}) we have
\begin{equation} \label{94}
(x \cdot (y \cdot (z \cdot  u)) / y = x \cdot (u \cdot z)
\end{equation}
Indeed, we can change  in (\ref{91}) $z\rightarrow uz$ and apply to the left side of obtained identity Tarki associative law (\ref{(ASS_TARKI)}).

>From (\ref{91}) and (\ref{94}) we have
$(x \cdot u) \cdot z = x \cdot (u \cdot z)$.
\end{proof}

\begin{example} \label{TAR_EX1}
\[
\begin{array}{ll}
\begin{array}{l|ll}
\cdot  & 0 & 1\\
\hline
    0 & 0 & 0 \\
    1 & 1 & 1
\end{array}
&
\begin{array}{r|rr}
/  & 0 & 1\\
\hline
    0 & 0 & 0 \\
    1 & 1 & 1
\end{array}
\end{array}
\]
\end{example}

As follows from Example \ref{TAR_EX1}, if a right division groupoid $(Q, \cdot , /)$ satisfies the Tarki associative  law (\ref{(ASS_TARKI)}), then it is not a  commutative groupoid and it does not contain two-sided identity element.

\subsection*{\centerline{Right division left cancelation groupoid}}

\begin{lemma} \label{TARKI_88}
If a left cancelation groupoid $(Q, \cdot , \backslash)$ satisfies the Tarki associative  law (\ref{(ASS_TARKI)}), then it is associative.
\end{lemma}
\begin{proof}
Let's rewrite (\ref{(ASS_TARKI)}) as follows
\begin{equation} \label{TARKI88_1}
(x \cdot y) \cdot z = x \cdot (z \cdot y)
\end{equation}
Considering (\ref{TARKI88_1}) as follows
\begin{equation} \label{TARKI88_4}
((x \cdot y) \cdot z) \cdot u = (x \cdot (z \cdot y)) \cdot u
\end{equation}
From left side of the equality (\ref{TARKI88_4}) we obtain
\begin{equation} \label{TARKI88_5}
((x \cdot y) \cdot z) \cdot u \overset{(\ref{TARKI88_1})}{=} (x \cdot y) \cdot (u \cdot z) \overset{(\ref{TARKI88_1})}{=} x \cdot ((u \cdot z) \cdot y)
\end{equation}
From right side of the equality (\ref{TARKI88_4}) we obtain
\begin{equation} \label{TARKI88_6}
(x \cdot (z \cdot y)) \cdot u \overset{(\ref{TARKI88_1})}{=} x \cdot (u \cdot (z \cdot y))
\end{equation}
and finally from (\ref{TARKI88_5}) and (\ref{TARKI88_6}) we have
\begin{equation} \label{TARKI88_7}
x \cdot ((u \cdot z) \cdot y) = x \cdot (u \cdot (z \cdot y))
\end{equation}
Applied (\ref{(3)}) to the (\ref{TARKI88_7}) we obtain
\begin{equation}
x \backslash (x \cdot ((u \cdot z) \cdot y)) \overset{(\ref{(3)})}{=} (u \cdot z) \cdot y = x \backslash (x \cdot (u \cdot (z \cdot y))) \overset{(\ref{(3)})}{=} u \cdot (z \cdot y)
\end{equation}
As a result, we have
\begin{equation}
(u \cdot z) \cdot y = u \cdot (z \cdot y)
\end{equation}
\end{proof}

\begin{lemma} \label{TARKI_888}
If a left cancelation groupoid $(Q, \cdot , \backslash)$ satisfies the Tarki associative  law (\ref{(ASS_TARKI)}), then $x\cdot y = y \cdot x$ for all $x,y \in Q$.
\end{lemma}
\begin{proof}
If we re-write Tarki identity in the form
$(x  y)  z = x  (z y)$ ($y\leftrightarrow z$) and take into consideration usual associative identity $(xy)z = x(yz)$, then
\begin{equation} \label{159}
x  (y  z) = x  (z  y)
\end{equation}
Further form (\ref{159}) we have
\begin{equation} \label{160}
x \backslash (x  (y  z)) = x \backslash (x  (z  y))
\end{equation}
and after applying to the both  sides of equality  (\ref{160}) identity \ref{(3)},
we have $y\cdot z = z\cdot y$.
\end{proof}

\begin{lemma} \label{TARKI_8811}
If a left cancelation right division groupoid $(Q, \cdot , \backslash, \slash )$ satisfies the Tarki associative  law (\ref{(ASS_TARKI)}), then  $(x \slash  x)  y = y$ for all $x, y \in Q$.
\end{lemma}
\begin{proof}
From identities $xy=yx$ and $(x \slash y) y = x$ (identity \ref{(2)}) we have
\begin{equation} \label{137}
x(y\slash x) = y
\end{equation}
If we change in identity (\ref{(3)}) $x \rightarrow y\slash x$, then we obtain $x \backslash (x(y \slash x)) = y\slash x$. If we apply identity (\ref{137}) to the left side of the last equality,
then we obtain
\begin{equation} \label{138}
x\backslash y = y \slash x
\end{equation}
In identity of associativity $(xy)z = x(yz)$ we change $y$ by $y\slash x$ and obtain $(x(y\slash x))z = x((y\slash x)z)$. After applying to the left side of the last equality identity (\ref{137}) we have
\begin{equation} \label{140}
yz = x((y\slash x)z)
\end{equation}
If we change  in (\ref{138}) $y$ by $xy$, then  $x\backslash (xy) = (xy) \slash x$,
\begin{equation} \label{141}
y = (xy) \slash x
\end{equation}
 since $x\backslash (xy) \overset{(\ref{(3)})}{=} y$.
From (\ref{140}) we have $(yz)\slash x = (x((y\slash x)z))\slash x$. But
\begin{equation*}
(x((y\slash x)z))\slash x \overset{(\ref{141})}{=}(y\slash x)z
\end{equation*}
Therefore
\begin{equation} \label{142}
(yz)\slash x = (y\slash x)z
\end{equation}
If we put in (\ref{142}) $x=y$, then $(xz)\slash x = (x\slash x)z$, $z = (x\slash x)z$, since by
(\ref{141}) $(xz)\slash x = z$.
\end{proof}

\begin{lemma} \label{TARKI_882}
If a left cancelation right division groupoid $(Q, \cdot , \backslash, \slash )$  satisfies the Tarki associative  law (\ref{(ASS_TARKI)}), then $x\slash x = y \slash  y$ for all $x, y \in Q$.
\end{lemma}
\begin{proof}
Since by Lemma \ref{TARKI_888} the groupoid $(Q, \cdot , \backslash, \slash )$ is commutative, then from 
  $(x \slash  x)  y = y$ we have $y (x \slash  x)   = y$. 
  
  If in the last equality we change $y$ by $y/y$, then we have $$(y/y)\cdot (x \slash  x)   = y/y$$
      Since the element $y/y$ is left identity element, then we have $x \slash  x   = y/y$.
\end{proof}

\begin{theorem} \label{TARKI_THEOR_9}
If a right division left cancelation groupoid $(Q, \cdot, \backslash, /)$ satisfies the Tarki associative  law (\ref{(ASS_TARKI)}), then it is a commutative group relative to the operation $\cdot$ and it satisfies any from identities (\ref{(6)}) -- ( \ref{(ASS_CIRCLIC)}).
\end{theorem}
\begin{proof}
The proof follows from Lemmas \ref{TARKI_88},  \ref{TARKI_888},  \ref{TARKI_8811} and  \ref{TARKI_882}.
\end{proof}

\noindent
\footnotesize{Institute of Mathematics and Computer Science \\
Academy of Sciences of Moldova\\
5 Academiei str.\\
  Chi\c{s}in\u{a}u  MD$-$2028 \\
 Moldova\\
e-mail: dmitry.pushkashu@gmail.com }


\begin{thebibliography}{20}

\bibitem{VD}
{\bf V.D. Belousov}: {\em Foundations of the Theory of Quasigroups and Loops}, Nauka, Moscow, (1967). (in Russian).

\bibitem{1a}
{\bf V.D. Belousov}: {\em Elements of Quasigroup Theory: a special course}, Kishinev State University Printing House, Kishinev, (1981) (in Russian).

\bibitem{BIRKHOFF}
{\bf G.~Birkhoff}: {\em Lattice Theory}, Nauka, Moscow, (1984) (in Russian).

\bibitem{HOP}
{\bf H.O. Pflugfelder}: {\em Quasigroups and Loops: Introduction}, Heldermann Verlag, Berlin, (1990).

\bibitem{BURRIS}
{\bf S.~Burris and H.P. Sankappanavar}: {\em A Course in Universal Algebra}, Springer-Verlag, (1981).

\bibitem{EVANS_51}
{\bf T.~Evans}: {\em On multiplicative systems defined by generators and relations}, Math. Proc. Camb. Phil. Soc., (1951) \textbf{47} $637-649$.

\bibitem{HOSSUZU}
{\bf M.~Hosszu}: {\em Some functional equations related with the associative law}, Publ. Math. Debrecen, (1954) \textbf{3} $205-214$.

\bibitem{KAZIM}
{\bf M. A. Kazim and M. Naseeruddin}: {\em On almost-semigroups}, The Alig Bull. Math., (1972) \textbf{2}, $1-7$.

\bibitem{PROTIC}
{\bf P. V. Protic and M. Bozinovic}: {\em Some congruences on an AG-groupoid},
Algebra Logic and Discrete Mathematics, (1995) \textbf{14-16}, $879-886$.

\bibitem{21}
{\bf J.~Je{\u z}ek and T.~Kepka}: {\em Medial groupoids}, se\u sit 2 of Rozpravy \u Ceskoslovenske Aca\-de\-mie V\u ED, Academia, Praha., (1983) volume $93$.

\bibitem{JKN}
{\bf J.~Je{\u z}ek, T.~Kepka, and P.~Nemec}: {\em Distributive groupoids}, se\u sit 3 of Rozpravy \u Ceskoslovenske Aca\-de\-mie V\u ED, Academia, Praha, (1981) volume 91.

\bibitem{MALTSEV}
{\bf A.I. Mal'tsev}: {\em Algebraic Systems}, Nauka, Moscow, (1976) (in Russian).

\bibitem{MAC_CUNE_PROV}
{\bf W.~McCune}: {\em Prover 9},
 University of New Mexico, www.cs.unm.edu/mccune/prover9/, 2007.

\bibitem{MUFANG}
{\bf R.~Moufang}: {\em \protect{Zur Structur von Alternativ K\" orpern}}, Math. Ann., (1934) \textbf{110} $416-430$.


\bibitem{PROT_STEP}
{\bf P.V. Protic and N. Stevanovic}: {\em \protect{Abel-Grassmann  bands}}, Quasigroups and Relat. Syst., (2004) \textbf{11} $95-101$.


\bibitem{SCERB_07}
{\bf V.A. Shcherbacov}: {\em On definitions of groupoids closely connected with quasigroups}, Bul. Acad. Stiinte Repub. Mold., Mat., (2007) no. $2$ $43-54$.

\bibitem{SCERB_TAB_PUSH_09}
{\bf V.A. Shcherbacov, A.Kh. Tabarov, and D.I. Pushkashu}: {\em On congruences of groupoids closely connected with quasigroups}, Fundam. Prikl. Mat., (2008) \textbf{14}(1) $237-251$.

\bibitem{SPS_10}
{\bf V.A. Shcherbacov, D.I. Pushkashu, A.V. Shcherbacov}: {\em Equational quasigroup definitions},
arXiv:1003.3175, 4 pages, http://arxiv.org/pdf/1003.3175v1.

\end{thebibliography}
\end{document}